\documentclass[twoside]{article}
\usepackage[left=2cm,right=2cm,top=2cm,bottom=2cm]{geometry}
\usepackage{amsmath}
\usepackage{amsfonts}
\usepackage{amssymb}
\usepackage{graphicx}
\usepackage{color}
\usepackage[normalem]{ulem}
\usepackage[utf8]{inputenc}
\usepackage{amsfonts}
\usepackage{color}
\usepackage[normalem]{ulem}
\newenvironment{proof}{\noindent{\sc Proof.}}{\qed}
\newtheorem{theorem}{Theorem}[section]
\newtheorem{lemma}{Lemma}[section]

\newtheorem{rem}{Remark}[section]
\newtheorem{definition}{Definition}[section]
\newtheorem{prop}{Proposition}[section]
\newtheorem{uda}{Example}[section]
\newcommand{\qed}{$\blacksquare$}

\def\bhag#1{\noindent
\setcounter{equation}{0}
\section{#1}
}

\def\RR{{\mathbb R}}

\def\ZZ{{\mathbb Z}}
\def\SS{{\mathbb S}}

\def\TT{\mathbb T}

\def\x{\mathbf{x}}
\def\k{\mathbf{k}}
\def\y{\mathbf{y}}

\def\w{\mathbf{w}}

\def\O{{\cal O}}

\def\C{{\mathcal C}}

\def\derf#1#2{{#1}^{(#2)}}

\def\esssup{\mathop{\hbox{\textrm{ess sup}}}}

\def\be{\begin{equation}}
\def\ee{\end{equation}}
\def\bea{\begin{eqnarray}}
\def\eea{\end{eqnarray}}
\def\eref#1{(\ref{#1})}
\def\disp{\displaystyle}

\def\donchitre#1#2{\vskip 6.5cm\noindent
\parbox[t]{1in}{\special{eps:#1.eps x=6.5cm y=5.5cm}}
\hbox to 7cm{}\parbox[t]{0.0cm}{\special{eps:#2.eps x=6.5cm y=5.5cm}}}

\def\tn{|\!|\!|}
\def\XX{{\mathbb X}}

\def\slp{{{\raise 0.5pt \hbox{\footnotesize $($}}}}   
\def\srp{{{\raise 0.5pt \hbox{\footnotesize $)$}}}}   
\newcommand{\ofp}[1]{{\slp{#1}\srp}} 

\title{
Super-resolution meets machine learning: approximation of measures }
\author{
 H.~N.~Mhaskar\thanks{
Institute of Mathematical Sciences, Claremont Graduate University, Claremont, CA 91711. The research of this author is supported in part   by the Office of the Director of National Intelligence (ODNI), Intelligence Advanced Research Projects Activity (IARPA), via 2018-18032000002.
\textsf{email:} hrushikesh.mhaskar@cgu.edu}
 }
 \date{}
\begin{document}

\maketitle

\begin{abstract}
The problem of super-resolution in general terms is to recuperate a finitely supported measure $\mu$ given finitely many of its coefficients $\hat{\mu}(k)$ with respect to some orthonormal system. The interesting case concerns situations, where the number of coefficients required is substantially smaller than a power of the reciprocal of the minimal separation among the points in the support of $\mu$.

In this paper, we consider the more severe problem of recuperating $\mu$ approximately without any assumption on $\mu$ beyond having a finite total variation. In particular, $\mu$ may be supported on a continuum, so that the minimal separation among the points in the support of $\mu$ is $0$.
A variant of this problem is  also of interest in machine learning as well as the inverse problem of de-convolution.

We define an appropriate notion of a distance between the target measure and its recuperated version, give an explicit expression for the recuperation operator, and estimate the distance between $\mu$ and its approximation.
We show that these estimates are the best possible in many different ways.

We also explain why for a finitely supported measure the 
approximation quality of its recuperation is bounded from below if the amount of information is smaller than what is demanded in the super-resolution problem.
\end{abstract}

\bhag{Introduction}\label{intsect}

This paper is motivated by two apparently disjoint areas; super-resolution and machine learning.
A problem of interest in both of these areas is the approximation of a measure using a finite amount of information on the measure.
Thus we wish to develop 
a theory of (weak-star) approximation of measures.
We will describe our motivation and the connections of this work to the problem of super-resolution and the problem of machine learning in Sections~\ref{super_res_sect} and ~\ref{machine_sect} respectively.
The aims and contributions of this paper, and its outline is given in Section~\ref{contributesect}.
This section and the next being introductory in nature, the notation used in these two sections may not be the same as the one used in the remainder of the paper.

\subsection{Super-resolution}\label{super_res_sect}
The problem of \emph{super-resolution}
is stated by Donoho \cite{donoho1992superresolution} as follows.
Given observations of the form
\be\label{donohoobs}
\sum_{k\in\ZZ}a_k\exp(-i\omega k\Delta) +z(\omega), \qquad |\omega|\le \Omega,
\ee
where  $\{a_k\}$ is  sequence  of complex numbers,  $\Delta, \Omega>0$, and $z$ represents a perturbation subject to the condition that $\int_{-\Omega}^\Omega |z(\omega)|^2d\omega \le \epsilon$,
 recuperate the sequence $\{a_k\}$ up to an accuracy $\O(\epsilon)$ in the sense of optimal recovery, when $\Omega\Delta<\pi$.
It is shown in \cite{donoho1992superresolution} that this is not possible in general, but it is possible under some sparsity assumptions.

Relevant to the current paper is a generalization stated by Cand\'es and Fernandez-Granda in \cite[Theorem~1.2]{candes2013super}. We denote the quotient space $\RR/(2\pi\ZZ)$ by $\TT$.
Let
\be\label{candeslow}
\mu(t)=\sum_{k=1}^K a_k\delta(t-t_k)+ Z(t),
\ee
where $\delta$ denotes the Dirac delta measure at $0$.
We assume that the moments 
$$
\sum_{k=1}^K a_k\exp(-ijt_k)+\int_\TT Z(t)\exp(-ijt)dt,
$$
are known for $|j|\le n$ for some integer $n\ge 1$.
The goal is to estimate how much degradation to expect when this data is extended for the values of $j$ with $n<|j|\le N$ for  some larger integer $N$.
Since the degradation is expected to be greater as $|j|$ increases, the authors propose to measure this degradation using the Fej\'er kernel
$$
F_N(t)=\frac{1}{N+1}\sum_{|j|\le N} (1-|j|/N)\exp(ijt).
$$
They prove that if
\be\label{candescond}
\eta=\min_{j\not=k}|t_j-t_k| \ge 2/n,
\ee
then a measure $\nu_n$ obtained by the solution of an optimization problem satisfies
\be\label{candesest}
\int_\TT \left|\int_\TT F_N(x-t)(d\mu(t)-d\nu_n(t))\right|dx \le c(N/n)^2\int_\TT |Z(t)|dt,
\ee
where $c$ is a positive constant.
We note that the number of observations considered known is $2n+1$.
Thus, the condition \eref{candescond} is a lower bound on the amount of information required in terms of the minimal separation in order to guarantee a stable recovery as measured in \eref{candesest}.

There is vast amount of literature on this problem,  where the problem is referred to with different names :
problem of hidden periodicities (e.g., \cite[Chapter~IV, Section~22]{lanczos}), direction finding in phased array antennas (e.g., \cite{krim}), detection of singularities (e.g., \cite{gelbtad, eckhof1, tanner1}), parameter estimation in exponential sums (e.g. \cite{singdet, pottstaschesingdet}), etc. The oldest we are aware of is the paper \cite{prony_original} of Prony, where the problem is considered without noise. We note also that there is some effort \cite{loctrigwave, batenkov2013accuracy, batenkov_stability_2016} in the direction of overcoming this barrier in the case of univariate trigonometric setting, where the information is in the form
\be\label{stackedobs}
\sum_{k=1}^K \sum_{r=0}^R a_{k,r}(-ij)^r\exp(-ijt_k), \qquad |j| <N,
\ee
so that each $t_k$ appears with multiplicity $R+1$.

We will not even begin to list the many modern works in the context of the periodic problem.
The problem has been studied also in other settings; for example, the sphere (e.g., \cite{bendory2015exact, bendory2015super}) or the rotation group \cite{filbir2016exact}, where the exponential monomials are replaced by eigenfunctions of the Laplace-Beltrami operator.
In all this work, it is required that the number of Fourier coefficients known about the measure is at least a constant multiple of $\eta^{-1/q}$, where $q$ is the dimension of the space involved.

This sort of condition seems to be an inherent barrier, we will refer to it as the \emph{minimal separation barrier}. 
We note that any finite set of points will have a positive minimal separation; the condition refers to the amount of information necessary to recuperate the measure to given accuracy.

In this paper, we are interested in approximating an arbitrary measure; not just a measure supported on a finite set.
When the support of the target measure is a continuum, then the minimal separation is $0$, and any recuperation with a finite amount of noisy data is arguably beyond super-resolution.
Of course, even in the absence of noise, an exact recuperation cannot be expected in general using only a finite amount of information.
On the other hand,
an approximate recovery is very standard in the trigonometric case; the entire section \cite[Chapter~2, Section~8]{zygmund} describes already many constructions, some of which are used in \cite{gelbtad} in the case of finitely supported measures.

To summarize, the question of overcoming the minimal separation barrier in super-resolution problems for point masses can be viewed as the problem of efficient approximation of a measure given a finite amount of information about the measure.

\subsection{Machine learning}\label{machine_sect}

A central problem in machine learning is to find a \emph{target function} $f$ on a space $\XX$, equipped with a probability measure $\mu^*$,  given  the information $f(x_j)+\epsilon_j$, $j=1,\cdots, M$, for some   points $x_j\in\XX$ chosen randomly from $\mu^*$, where $\epsilon_j$ is a random noise.
Typically, the function $f$ is defined on a very high dimensional space, say a space of dimension $Q$.
In this case, there are well known results in approximation theory, known as width theorems, which give a lower bound of the form $M^{-\gamma/Q}$ on how accurately one can approximate a function, for which the only available a priori information is that  it belongs to a smoothness class indexed by $\gamma$ (e.g., a Sobolev class) \cite{devore1989optimal}.
This is known as the \emph{curse of dimensionality}.

In recent years, deep networks have caused a revolution in machine learning,  with many spectacular achievements in industrial problems.
It is therefore an important problem to examine why and when deep networks perform better than the so called shallow networks.
We have argued in \cite{dingxuanpap} that one reason that deep networks perform better than shallow networks is that many functions of practical interest have a compositional structure which deep networks can exploit and shallow networks cannot.
For example, suppose that the target function $f$ is known to have the structure
\be\label{composition}
f (x_1,\cdots,x_4)=
f_1\left (f_{12}\ofp{x_1,x_2}, f_{34}\ofp{x_3, x_4} \right),
\ee
where the functions $f_1$, $f_{12}$, $f_{34}$ are continuously differentiable on a cube $[-1,1]^2$. A shallow network of the form
$$
\sum_{k=1}^N a_k\sigma(\w_k\cdot\x+b_k), \quad a_k, b_k \in\RR, \w_k\in\RR^4
$$
where $\sigma$ is a suitable activation function, 
yields an approximation $O(N^{-1/4})$ \cite{optneur}.

However, if we use the same construction as in \cite{optneur} to obtain shallow networks $P_1, P_{11}, P_{12}$, each with $N$ terms to approximate $f_1, f_{12}, f_{13}$ respectively,  one gets an accuracy $O(N^{-1/2})$ in each approximation. By the triangle inequality 
 the deep network given by
$$
P(x_1,\cdots,x_4)= P_1 \left (P_{12}\ofp{x_1,x_2}, P_{34}\ofp{x_3,x_4} \right)
$$
yields an accuracy $O(N^{-1/2})$, using only $O(N)$ parameters.

This theory suggests on the other hand that deep networks do not give any advantage when there is no curse of dimensionality.
There is some research exploring other prior assumptions on the target function which ensure that there is no curse of dimensionality of the kind described above.

Let us illustrate the situation by an example. 
Assume that $f$ admits a representation of the form
\be\label{rkhsl1def}
f(x)=\int_\XX G(x,y)d\mu(y), \qquad x\in\XX,
\ee
for some measure $\mu$ having a bounded total variation on $\XX$, then it is known that one can obtain approximations to $f$ by linear combinations of the form $\sum_{j=1}^n a_jG(x,y_j)$ with the degree of approximation being dimension-independent in terms of $n$, often tractable as well (e.g., \cite{barron1992neural, klusowski2016uniform, kurkova1, kurkova2, tractable}).
The total variation of $\mu$, sometimes known as the $G$-variation of $f$, plays the role of the norm of the derivatives of $f$ in classical approximation theory.
The proofs of such theorems depend upon a probabilistic argument, and are not constructive.
Therefore, in practice, the parameters $a_j$, $y_j$ are determined using some learning algorithm.

Thus, in theory, the problem is to determine $\mu$ given finitely many samples of $f$, even though the number of these samples may not be dimension independent.
Let us assume that $\XX$ is a compact Riemannian manifold on which $G$ has a Mercer expansion of the form $\sum_{k=0}^\infty \Lambda_k \phi_k(x)\phi_k(y)$ for some orthonormal system $\{\phi_k\}$.
Then  it is shown under some conditions in \cite{modlpmz} (see also \cite{compbio}) that one can obtain quadrature formulas to integrate all linear combinations of the first few functions $\{\phi_k\}$; the number of these functions dependent on the number of points at which the values of $f$ are available.
Thus,  in theory, the inverse problem of recuperating $\mu$ from the samples of $f$ is reduced to recuperating $\mu$ (respectively, its discretized version) using finitely many Fourier coefficients of $\mu$ with respect to the system $\{\phi_k\}$.

A question of theoretical interest here is  to estimate the degree of approximation of $\mu$ in a suitable sense in terms of the number of Fourier coefficients that can be computed reliably from the data.
This is the same problem that we were led to in our musings on super-resolution in Section~\ref{super_res_sect}, thereby establishing a close connection between these two problems.

\subsection{Contributions of this paper}\label{contributesect}
The problem of approximating a measure in the weak-star sense is inherently different from that of approximation of functions, and also from that of an exact or approximate recuperation of the support of the measure.
In particular, in the context of machine learning and probability density estimation, it is customary to use a ``convolution'' with a positive kernel. 
As an approximation device, it is well known that this is doomed not to give a good approximation. 
However, if we use a non-positive kernel to guarantee good approximation as we propose to do in this paper, then there are some difficulties in recuperating the support of the measure exactly and directly without some non-linear operations such as thresholding and clustering.
In this paper, we are focused on approximation of measures, and will postpone the discussion of these other issues to future work.

We will  describe how we address the difference between approximation of functions and that of measures in  the current paper.

To study an approximation problem, one needs the notion of a distance between two objects and some notion of smoothness of the target object to be approximated.
In classical theory of function approximation, there are standard ways of defining both of these; e.g., in approximating a continuous $2\pi$-periodic function by trigonometric polynomials, one uses the uniform norm and the smoothness is measured by an appropriate modulus of smoothness.
Such questions have been studied in many different contexts, e.g., \cite{devlorbk}.

In contrast, there is no standard definition for measuring the distance between two measures so that the convergence in the topology corresponds to weak-star convergence.
There are several ways of defining such a distance in different application domains, and we will list a few of these in Section~\ref{distsect} to motivate our own definition.
However, we are not aware of any standard smoothness class for measures.
In our Theorem~\ref{murecuptheo} below, we will observe that with no assumption on the target measure,  the degree of approximation depends entirely on the definition of the distance.
Since an estimate on the degree of approximation is typically (and in our theorem, is) achieved using a specific construction (given in \eref{recuperateop}), this surprising fact gives rise to the question whether one could obtain better estimates using a different construction, or even using a different kind of information about the measure.
We will discuss these issues in Theorem~\ref{widththeo} and Theorem~\ref{discmasstheo}.
In particular, Theorem~\ref{discmasstheo} provides one explanation for the minimal separation barrier in super-resolution of point masses as described in Section~\ref{super_res_sect}.
Another natural question to ask is to understand
 what a better estimate on the degree of approximation allows us to conclude about the measure, analogous to the converse theorems of approximation theory.
A trivial case is when the target measure is absolutely continuous with respect to a base measure, and the Radon-Nikodym derivative is then approximated as in the classical function approximation paradigm.
We will prove in Theorem~\ref{convtheo} that an improvement on the approximation bounds in Theorem~\ref{murecuptheo} implies that the target measure is in fact absolutely continuous with respect to a base measure and the derivative is in the right smoothness class as expected in the theory of function approximation.

In Section~\ref{distsect}, we review a few notions of distance between measures in order to motivate our Definition~\ref{distdef}.
In
Section~\ref{notationsect}, we develop the set up for our theory, and establish some notation to be used in the subsequent sections. The main results are discussed in Section~\ref{mainsect}, and the proofs of all the new results are given in Section~\ref{pfsect}.

We thank Professor Dr. Hans Feichtinger for many useful comments on the presentation in this paper.

\bhag{Distance between measures}\label{distsect}

In order to discuss  the quality of approximate recuperation of measures, we need first to develop a notion of distance between measures. There are many ways of defining  a distance.
We mention a few of these to motivate Definition~\ref{distdef} which we will use in this paper.

In the univariate case, a very old way to define a distance is the 
\emph{Erd\H os-Tur\'an discrepancy} (known also as \emph{Kolmogorov-Smirnov statistic} in statistics and \emph{star discrepancy} in information based complexity).
In the context of measures on $\TT=\RR/(2\pi\ZZ)$ (identified with 
$[-\pi,\pi)$),  this is defined for a signed measure $\nu$ with $\nu(\TT)=0$ by
\be\label{erdosturan}
D^{ET}(\nu)=\sup_{x\in [-\pi,\pi)}|\nu([-\pi,x))|.
\ee
A comparison of Fourier coefficients shows that
\be\label{erdosturankern}
D^{ET}(\nu)=\left\|\int_\TT G(\circ-t)d\nu(t)\right\|_{\TT,\infty},
\ee
where $G$ denotes the \emph{Bernoulli spline} defined by $\disp G(u)=\sum_{k=1}^\infty \frac{\sin ku}{k}$, $u\in [-\pi,\pi)$.
In the form \eref{erdosturankern}, this notion of discrepancy is generalized using many different kernels on different high dimensional domains (e.g., \cite{novak2008tractability, dick2010digital}).
A similar notion in statistics is the so called \emph{maximum mean discrepancy} (MMD), defined by
$$
\left\|\int_\XX G(\circ,t)d\nu(t)\right\|_{\XX,2}
$$
for some measure space $\XX$ and a positive definite kernel $G$ defined on this space.

Another popular distance between measures is the $L^1$-\emph{Wasserstein distance}. Let $\XX$ be a metric space and $\nu$ be a Borel measure on this space with $\nu(\XX)=0$. One of the equivalent definitions of this distance is given by
$\disp
\sup\left|\int_\XX fd\nu\right|,
$
where the supremum is over all Lipschitz continuous functions $f$ on $\XX$ with Lipschitz constant $\le 1$. If $\XX$ is a manifold, $\Delta$ is the Laplace-Beltrami operator on $\XX$, an analogue of this distance, more responsive to the manifold structure, is obtained by taking the supremum over all functions with $\|\Delta(f)\|_{\XX, \infty}\le 1$. Denoting the Green function for $\Delta$ by $G$, this in turn is equivalent to
$$
\left\|\int_\XX G(\circ,t)d\nu(t)\right\|_{\XX,1}.
$$

Finally, we note that the estimate \eref{candesest} utilizes a semi-norm of the form
$$
\left\|\int_\TT F_N(\circ-t)d\nu(t)\right\|_{\TT,1}
$$
where $F_N$ is the Fej\'er kernel with Fourier coefficients equal to $0$ outside of $[-N,N]$.

\bhag{Notation and definitions}\label{notationsect}

In this section, we describe the general set up for our discussion, and establish notation.

Let $\XX$ be a locally compact metric measure space, with
$d$ denoting the metric on $\XX$, and
$\mu^*$ being a distinguished positive measure on $\XX$.
In the sequel, 
only complete, sigma finite, Borel measures are considered, 
defined on a sigma algebra $\mathfrak{M}$ 
containing all Borel subsets of $\mathbb{X}$. In the sequel, 
$\nu$-measurability will be understood in the sense of membership
in this fixed sigma algebra. 

For  $B\subseteq \mathbb{X}$, $\nu$-measurable, and  a $\nu$-measurable function $f : B\to \RR$  we   write
$$
\|f\|_{\nu;B,p}:=\left\{\begin{array}{ll}
\displaystyle \left\{\int_B|f(x)|^pd|\nu|(x)\right\}^{1/p}, & \mbox{ if $1\le p<\infty$,}\\
\displaystyle|\nu|-\esssup_{x\in B}|f(x)|, &\mbox{ if $p=\infty$.}
\end{array}\right.
$$
$L^p(\nu;B)$   denotes the class of all $\nu$--measurable functions $f$ for which $\|f\|_{\nu;B,p}<\infty$, where two functions are considered equal if they are equal $|\nu|$--almost everywhere. We will omit the mention of $\nu$ if $\nu=\mu^*$ and that of $B$ if $B=\mathbb{X}$. Thus, $L^p=L^p(\mu^*;\mathbb{X})$.
 For $1\le p\le\infty$, we define $p'=p/(p-1)$ with the usual understanding that $1'=\infty$, $\infty'=1$.
The symbol $C_0(B)$ denotes the space of all continuous real functions on $B$ vanishing at infinity; $C_0=C_0(\XX)$.
 The symbol $C_0^*$ will denote the dual space of $C_0$; i.e., the class of all regular, Borel, measures with bounded total variation.

We also need a non-decreasing sequence $\{\lambda_k\}_{k=0}^\infty$ of real numbers, and an ($\mu^*$-) orthonormal system of functions $\{\phi_k\}_{k=0}^\infty$ in $C_0(\XX)\cap L^1(\XX)$. We assume that
$\lambda_0=0$, and $\lim_{k\to\infty}\lambda_k=\infty$. In addition we assume that the system $\{\phi_k\}_{k=0}^\infty$ is fundamental in both $L^1$ and $C_0$.

\begin{definition}\label{ddrdef}
The system $\Xi=(\mathbb{X},d,\mu^*,\{\lambda_k\}_{k=0}^\infty,\{\phi_k\}_{k=0}^\infty)$ is called an \textbf{admissible system} if
\begin{enumerate}

\item For each $x\in\mathbb{X}$ and $r>0$, the ball $\mathbb{B}(x,r)$ is compact.
\item There exists $q>0$ and $\kappa_1, \kappa_2, \kappa_3>0$ such that for $x\in\mathbb{X}$,  $r>0$,
\be\label{ballmeasurecond}
\mu^*(\mathbb{B}(x,r))=\mu^*\left(\{y\in\mathbb{X} : d(x,y)<r\}\right) \le \kappa_1r^q.
\ee
\item  For  $x, y\in \mathbb{X}$, $0<t\le 1$,
\be\label{gaussupbd}
\left|\sum_{k=0}^\infty \exp(-\lambda_k^2t)\phi_k(x){\phi_k(y)}\right|\le \kappa_2t^{-q/2}\exp\left(-\kappa_3\frac{d(x,y)^2}{t}\right)
\ee
\end{enumerate}
\end{definition}

\begin{rem}\label{ddsrem}
{\rm
In some of our other papers we have referred to an \emph{admissible system} in the sense of the above definition as a \emph{data defined space}.
This is motivated  by an idea for semi-supervised learning, called \emph{diffusion geometry/manifold learning}.
One assumes that the data for this kind of machine learning problem lives on an unknown low dimensional sub-manifold of a high dimensional Euclidean space.
The learning takes place based on the eigen-decomposition of a suitably constructed graph Laplacian. In theory, one may assume the eigen-decomposition of the heat kernel with respect to an elliptic differential operator on the manifold itself.
The properties of this heat kernel play a central role in the theoretical development.
In particular, it is shown in \cite[Theorem~4.3]{tauberian} that the condition \eref{gaussupbd} implies the localization properties of the kernels $\Phi_n$ defined in \eref{lockerndef} below; which in turn, plays a crucial role in this paper via Proposition~\ref{eignetprop}.
\qed
}
\end{rem}
\noindent\textbf{Constant convention:}\\

In the sequel, the symbols $c, c_1,\cdots$ will denote generic positive constants depending only on the system $\Xi$ and other constant parameters under discussion. Their value will be different at different occurrences, even within a single formula. The notation $A\sim B$ means $cA\le B\le c_2B$. \qed\\

We now define a candidate for a semi-norm on $C_0^*$ which will be used in this paper.
\begin{definition}\label{distdef}
Let $G :\XX\times\XX\to\RR$ be a kernel that admits a formal Mercer expansion $\disp G(x,y)\!=\!\sum_{j=0}^\infty b(\lambda_j)\phi_j(x)\phi_j(y)$, where $b(\lambda_j) \ge 0$ for every $j\ge 0$. For $\mu\in C_0^*$ and $1\le p\le\infty$, we define formally
\be\label{mumetric}
\tn\mu\tn_{G;p}=\left\|\int_{\XX} G(\circ,y)d\mu(y)\right\|_p.
\ee
\end{definition}

We will be particularly interested in the following class of kernels 
(cf. \cite{eignet}):
\begin{definition}\label{eigkerndef}
Let $\beta\in\RR$. A function $b:\RR\to\RR$ will be called a mask of type $\beta$ if $b$ is an even, $S$ times continuously differentiable function such that for $t>0$, $b(t)=(1+t)^{-\beta}F_b(\log t)$ for some $F_b:\RR\to\RR$ such that  $|\derf{F_b}{k}(t)|\le c(b)$, $t\in\RR$, $k=0,1,\cdots,S$,  and $F_b(t)\ge c_1(b)$, $t\in\RR$.     A function $G:\XX\times\XX\to \RR$ will be called  a kernel of type $\beta$ if it admits a formal expansion $G(x,y)=\sum_{j=0}^\infty b(\lambda_j)\phi_j(x)\phi_j(y)$ for some mask $b$ of type $\beta>0$. If we wish to specify the connection between $G$ and $b$, we will write $G(b;x,y)$ in place of $G$.
\end{definition}

\begin{uda}\label{trigkernexample}
{\rm
We consider $\XX=\TT^q$.  If $\beta>0$, then the kernel defined formally by
$$
\sum_{\k\in\ZZ^q}(\|\k\|^2+1)^{-\beta/2}\exp(i\k\cdot\circ)
$$
is a kernel of type $\beta$.
\qed}
\end{uda}
\begin{uda}\label{spherekernexample}
{\rm
We consider the unit sphere $\SS^q=\{\x\in\RR^{q+1} : |\x|_2=1\}$. If $\beta>0$, the kernel defined formally by $G(\x,\y)=(1-\x\cdot\y)^{(\beta-q)/2}$ is a kernel of type $\beta$ (\cite[Section~9.3(4)]{szego}).
\qed}
\end{uda}

When $G$ is a kernel as defined in Definition~\ref{eigkerndef}, $\tn\circ\tn_{G;p}$ is a norm
consistent with the weak-star topology on $C_0^*$. We will give a proof of the following simple proposition in Section~\ref{pfsect}.

\begin{prop}\label{mumetricprop}
Let $1\le p\le \infty$, $\beta>q/p'$, $G$ be a kernel of type $\beta$. Then  the functional $\tn\circ\tn_{G;p}$ defines a norm on $C_0^*$. If $\{\nu_n\}_{n=0}^\infty$ is a sequence in $C_0^*$, $\nu\in C_0^*$ then $\nu_n\stackrel{*}{\to}\nu$ if and only if $\tn\nu_n-\nu\tn_{G;p}\to 0$.
\end{prop}
Next, we define some \emph{ smoothness classes of functions} in terms of their degree of approximation by linear combinations of  $\{\phi_k\}$.
We define
\be\label{diffpolydef}
\Pi_\lambda=\mathsf{span}\{\phi_k : \lambda_k <\lambda\}, \qquad \lambda>0,
\ee
and $\disp\Pi_\infty=\bigcup_{\lambda>0}\Pi_\lambda$. If $\lambda\le 0$, we denote $\Pi_\lambda=\{0\}$. Following \cite{mauropap}, we refer to the elements of $\Pi_\infty$ as \emph{diffusion polynomials}. The $L^p$-closure of $\Pi_\infty$ is denoted by $X^p$; i.e., $X^p=L^p$ if $1\le p<\infty$, and $C_0$ if $p=\infty$.

If $1\le p\le\infty$ and $f\in L^p$, we define
\be\label{degapproxdef}
E_{\lambda,p}(f)=\inf_{P\in\Pi_\lambda}\|f-P\|_p, \qquad \lambda\in\RR.
\ee
If $r>0$ then the smoothness class $W_{r,p}$ is the set of all $f\in X^p$ such that
\be\label{sobolevnormdef}
\|f\|_{W_{r,p}}=\|f\|_p +\sup_{n\ge 0}2^{nr}E_{2^n,p}(f) <\infty.
\ee

Our main tool in the recuperation of measures is a localized kernel.
Given  a compactly supported function $H :\RR\to\RR$, we define: 
\be\label{lockerndef}
\Phi_n(H;x,y):=\sum_{k=0}^\infty H\left(\frac{\lambda_k}{n}\right)\phi_k(x)\phi_k(y), \qquad n>0, \ x,y\in\XX.
\ee
For $\mu\in C_0^*$, we define formally
\be\label{sigmaopdef}
\sigma_n(H;\mu)(x)=\int_\XX \Phi_n(H;x,y)d\mu(y), \qquad x\in\XX, \ n>0.
\ee
We write
\be\label{fourcoeffdef}
\hat{\mu}(k)=\int_\XX \phi_k(y)d\mu(y), \qquad k=0,1,\cdots.
\ee
Then
\be\label{sigmainfour}
\sigma_n(H;\mu)(x)=\sum_{k=0}^\infty H\left(\frac{\lambda_k}{n}\right)\hat{\mu}(k)\phi_k(x), \qquad n>0, \ x\in\XX.
\ee
We note that $\sigma_n(H;\mu)$ can be identified with the measure $\sigma_n(H;\mu)d\mu^*$. In general,
if $\mu$ is absolutely continuous, so that $d\mu=fd\mu^*$ for some $f\in L^1$, then by an abuse of the notation we write $\hat{f}(k)$ for $\hat{\mu}(k)$, and likewise, $\sigma_n(H;f)$ for $\sigma_n(H;\mu)$.

\bhag{Main results}\label{mainsect}

Our first objective is to estimate the degree of approximation in recuperating a measure $\mu\in C_0^*$ from noisy measurements of the form $\hat{\mu}(k)+\epsilon_k$, for $k$ with $\lambda_k<2^n$. Toward this end,  we fix in the rest of this paper, an infinitely differentiable, even function $h:\RR\to\RR$ such that $h$ is non-increasing on $[0,\infty)$, $h(t)=1$ if $0\le t\le 1/2$, $h(t)=0$ if $t \ge 1$. The constants $c, c_1,\cdots$ will depend upon $h$ as well.

The approximation to $\mu$ is the measure $\nu_n$, defined spectrally by
\be\label{recuperateop}
\widehat{\nu_n}(k)=h(\lambda_k/2^n)\left\{\hat{\mu}(k)+\epsilon_k\right\}, \qquad k=0,1,\cdots.
\ee
We find it convenient to denote the noiseless recuperation measure by $\mu_n$; i.e.,
\be\label{muapproxdef}
\mu_n(B)=\int_B\int_\XX \Phi_{2^n}(h; x,y)d\mu(y)d\mu^*(x)=\int_B \sigma_{2^n}(h;\mu)(x)d\mu^*(x), \qquad n\ge 0,
\ee
for all Borel subsets $B\subseteq \XX$.

The following theorem shows that the rate at which the degree of approximation of $\mu$ by $\{\nu_n\}$ (as a function of $n$), measured in the norm given in Definition~\ref{distdef}, decreases to $0$  depends only on the kernel $G$. There is no natural way to define a smoothness of the measure $\mu$.
\begin{theorem}\label{murecuptheo}
Let $1\le p \le \infty$, $\beta>q/p'$, $G$ be a kernel of type $\beta$, and $\mu\in C_0^*$, $\nu_n$ be defined by \eref{recuperateop}. Let $P_n=\sum_{k:\lambda_k<n}\epsilon_k\phi_k$. Then
\be\label{mudegapprox}
\tn\nu_n-\mu\tn_{G;p} \le c\left\{2^{-n(\beta-q/p')}|\mu|(\XX)+\|P_n\|_1\right\}.
\ee
Moreover, for the high pass filter $G_{hi}(x,y)=G(x,y)-\Phi_{2^n}(hb_{2^n};x,y)$, we have
\be\label{highpassapprox}
\tn\nu_n-\mu\tn_{G_{hi};p} \le c2^{-n(\beta-q/p')}\left\{|\mu|(\XX)+\|P_n\|_1\right\}.
\ee
\end{theorem}

\begin{rem}\label{candesrem}
{\rm
We compare this theorem with \cite[Theorem~1.2]{candes2013super} described in Section~\ref{super_res_sect}.
 The analogue of the  high pass filter $F_N$ is given by $G_{hi}$.
Note that, unlike \eref{candesest}, the noise term $\|P_n\|_1$  has a decreasing influence in the high pass range.
Analogous to the kernel $F_N$, the kernel $G_{hi}$ gives a lower weight to the higher frequencies, but
 unlike the kernel $F_N$, the kernel $G_{hi}$ includes \textbf{all} the high frequency components.

We note that there is no longer any assumption on the minimal separation among the points in the support of the target measure $\mu$.
An exact recovery is in general impossible, even in the noise-free case.
Our construction in \eref{recuperateop} being general, does not give an exact recuperation also in the case of finitely supported measures without some further processing, which is not within the scope of this paper.
However, the result is applicable for measures defined on a very general space, and does not require the verification of a signature polynomial as in \cite{candes2014_super_math_theory}.
Therefore, we expect that the approximation $\nu_n$ is easier to construct so as to obtain a good approximation, even if no  exact recovery is possible.
\qed}
\end{rem}

\begin{rem}\label{eignetrem}
{\rm
Let $f$ admit a representation of the form
$$
f(x)=\int_\XX G(x,y)d\mu(y)
$$
for some measure $\mu\in C_0^*$.
A comparison of Fourier coefficients shows that
$$
\sigma_{2^n}(f)(x)=\int_\XX G(x,y)d\mu_n(y).
$$
Therefore, Theorem~\ref{murecuptheo} implies 
$$
\|f-\sigma_{2^n}(f)\|_p \le c2^{-n(\beta-q/p')}|\mu|(\XX).
$$
In particular, in the case $p=1$, we get bounds nominally sharper than those in \cite{mauropap}. Rather than assuming a condition on $f$ in terms of pseudo-differential operators (informally, choosing $G$ to be a Green function of a pseudo-differential operator), we allow a more general kernel $G$. Also,  we no longer require the object $\mathcal{D}_G(f)$ defined spectrally by $\widehat{\mathcal{D}_G(f)}(k)=\hat{f}(k)/b(\lambda_k)$, $k=0,1,\cdots$, to be a function in $X^1$, but allow it to be a measure.
It is explained in \cite{eignet, compbio} how to discretize the quantity $\sigma_{2^n}(f)$ based on values of $f$ at scattered data points.
This leads to a constructive procedure to obtain an approximation to $f$ by sums of the form $\sum_{j=1}^M a_jG(\circ, y_j)$ \cite{eignet}. However, the error bounds are not dimension independent.
Dimension independent bounds can be obtained using concentration inequalities in a probabilistic sense, but then the proof is not constructive.
\qed}

\end{rem}
Next, we address the question whether one can improve upon the bounds in \eref{mudegapprox}. For simplicity, we consider the noiseless case; i.e., assume in the sequel that $P_n\equiv 0$.
The first theorem below states that one cannot improve the factor of $2^{-n(\beta-q/p')}$ to $2^{-n(\beta+r)}$ except in ``trivial'' cases; i.e., when $\mu=fd\mu^*$ for some $f\in W_{\gamma,p}$, so that results from function approximation are applicable directly.
Thus, in the case when $p=1$, the estimate \eref{mudegapprox} cannot be improved.

\begin{theorem}\label{convtheo}
Let $1\le p\le\infty$, $\beta>q/p'$, $r>0$, $G$ be a kernel of type $\beta$, $\mu\in C_0^*$ and  for each $n\ge 1$, $\mu_n$ be defined by \eref{muapproxdef}. Then the following are equivalent:\\
{\rm (a)} There exists $f\in W_{r,p}$ such that $d\mu=fd\mu^*$.\\
{\rm (b)} We have
\be\label{finedegapprox}
\sup_{n\ge 0}2^{n(\beta+r)}\tn\mu_n-\mu\tn_{G;p} <\infty.
\ee
\end{theorem}

Another way to examine a possible improvement in \eref{mudegapprox} is using the notion of \emph{non-linear widths}. We note that the recuperation measure $\mu_n$ depends upon the parameters $\hat{\mu}(k)$, for $k$ such that $\lambda_k<2^n$; i.e., as many parameters as the dimension of $\Pi_{2^n}$.
In most manifolds, the eigenfunctions $\phi_k$ of the Laplace-Beltrami operator satisfy an additional estimate given in \eref{phinlowbd} below (see \cite{frankbern, modlpmz} for a fuller discussion). In the general set up which we are working with, it is therefore reasonable to assume that there exists  $\gamma>0$ such that
\begin{equation}\label{phinlowbd}
\min_{y\in \mathbb{B}(x,\gamma/m)}|\Phi_m(x,y)| \ge cm^q, \qquad x\in \mathbb{X}, \ m\ge 1.
\end{equation}
Under this assumption, it is not difficult to verify that the dimension of $\Pi_{2^n}$ is $\sim 2^{nq}$.
Thus, in terms of the number $M$ of parameters used in the recuperation,
the bound \eref{mudegapprox} for the case $p=1$ is $\O(M^{-\beta/q})$.
We now proceed to show that this is the best possible.

Let $\mathcal{K}$ be a weak-star compact subset of $C_0^*$.
We denote by $\mathcal{S}$ the set of all weak-star continuous mappings from $\mathcal{K}\to\RR^M$ (parameter selection maps).
An algorithm is a mapping
 $A: \RR^M\to C_0^*$. Thus, for any algorithm $A$ and parameter selection $S$, and $\mu\in \mathcal{K}$,  $A(S(\mu))\in C_0^*$ is an attempted reconstruction of $\mu$ from the data $S$ using the algorithm $A$. We define
$$
\mbox{Err}_M(G,p;\mathcal{K},A,S)=\sup_{\mu\in \mathcal{K}}\tn A(S(\mu))-\mu\tn_{G;p},
$$
and the nonlinear width of $\mathcal{K}$ in the sense of $\tn\circ\tn_{G;p}$ by
\be\label{widthdef}
d_M(G,p;\mathcal{K})= \inf_{A: \RR^M\to C_0^*}\inf_{S\in\mathcal{S}}\mbox{Err}_M(G,p;\mathcal{K},A,S), \qquad M\ge 1.
\ee

\begin{theorem}\label{widththeo}
Let $\XX$ be compact, $\mu^*(\XX)=1$, $\beta>0$, $G$ be a kernel of type $\beta$. We assume further that there exists  $\gamma>0$ such that \eref{phinlowbd} holds. Let
\be\label{unitballdef}
\mathcal{K}=\{\nu\in C_0^* : |\nu|(\XX)\le 1\}.
\ee
Then for integer $M\ge 1$,
\be\label{widthest}
d_M(G,1;\mathcal{K}) \ge cM^{-\beta/q}.
\ee
\end{theorem}

We end this section with a width result that demonstrates that the minimal separation is an essential barrier to the recuperation of finitely supported measures, not just from the Fourier information, but from any robust parameter selection.
Toward this end, let $\eta>0$ and
\be\label{discmeasclass}
\mathcal{K}_\eta=\{\mu = \sum_j a_j\delta_{x_j}: a_j\in\RR,\ x_j\in \XX, \ \sum_j |a_j|\le 1,\ \min_{j\not=k}d(x_j,x_k)\ge \eta\}.
\ee
Although we do not prescribe the exact number of point masses in the definition above, when $\XX$ is compact, then a volume argument shows that this number cannot exceed $c\eta^{-q}$.
It is not difficult to show in this case that $\mathcal{K}_\eta$ is a compact subset of $C_0^*$.

\begin{theorem}\label{discmasstheo}
Let $\XX$ be compact, $\mu^*(\XX)=1$, $\beta>0$, $G$ be a kernel of type $\beta$. We assume further that \eref{phinlowbd} holds. Then for integer $M\sim\eta^{-q/\beta}$,
\be\label{discwidth}
d_M(G,1;\mathcal{K}_\eta)\ge c\eta.
\ee
\end{theorem}

\begin{rem}\label{discwidthrem}
{\rm
We remark that $d_m(G;\mathcal{K}_\eta)$ is a decreasing function of $m$. Therefore, the estimate \eref{discwidth} shows a lower limit on how accurately a finitely supported measure with the minimal separation of its support equal to $\eta$ can be approximated using $\le c\eta^{-q/\beta}$ continuously selected parameters.
\qed}
\end{rem}

\begin{rem}\label{pronyrem}
{\rm
In the case when $\XX=\TT$,  and $\mu$ is measure supported on $N$ points, then the Prony method can recuperate the measure exactly using $2N+1$ parameters, regardless of minimal separation among the points.
This is not a contradiction to Theorem~\ref{discmasstheo}, which refers to the worst case error for approximating measures in $\mathcal{K}_\eta$.
For any $\eta>0$, the class $\mathcal{K}_\eta$ contains a measure supported on $N\sim \eta^{-1}$ points and for this measure, $2N+1\sim  c\eta^{-1}$.\qed
}
\end{rem}

\bhag{Proofs}\label{pfsect}

In the sequel, if $N>0$, we will write $b_N(t)=b(Nt)$.  If $\C\subset\XX$ is a finite set, we define
\be\label{minsepdef}
\eta(\C)=\min_{x,y\in\C, x\not=y}d(x,y).
\ee
For $P\in\Pi_\infty$, we define
\be\label{Gderdef}
\mathcal{D}_GP(x)=\sum_k \frac{\hat{P}(k)}{b(\lambda_k)}\phi_k(x), \qquad x\in \XX,
\ee
so that
\be\label{PGinversion}
P(x)=\int_\XX G(x,y)\mathcal{D}_GP(y)d\mu^*(y), \qquad P\in \Pi_\infty, \ x\in\XX.
\ee
In the sequel, we write $g(t)=h(t)-h(2t)$, $t\in\RR$.

We recall the following results from \cite{eignet}. Although the set up there is  that of a compact smooth manifold without boundary, the proofs are verbatim the same for admissible spaces.

\begin{prop}\label{eignetprop}
Let $1\le p\le \infty$, $\beta\in\RR$, $b$ be a mask of type $\beta$.\\
{\rm (a)} We have
\be\label{btimesgnorm}
\sup_{x\in \XX}\|\Phi_{2^n}(g b_{2^n};x,\circ)\|_p \le c2^{-n(\beta-q/p')}, \qquad \sup_{x\in \XX}\|\Phi_{2^n}(g ;x,\circ)\|_p \le c2^{nq/p'}
\ee
{\rm (b)} If $\beta>q/p'$ then for every $y\in\XX$, there exists $\psi_y:=G(\circ,y)\in X^p$ such that $\widehat{\psi_y}(k) =b(\lambda_k)\phi_k(y)$, $k=0,1,\cdots$. We have
\be\label{glpuniform}
\sup_{y\in\XX}\|G(\circ,y)\|_p \le c, \quad \sup_{y\in\XX}\|G(\circ,y)-\Phi_{2^n}(hb_{2^n};\circ,y)\|_p \le c2^{-n(\beta-q/p')}.
\ee
{\rm (c)} If $\beta>0$, $n\ge 1$,  $P\in\Pi_n$ then
\be\label{bernstein}
\|\mathcal{D}_GP\|_p \le cn^\beta\|P\|_p.
\ee
{\rm (d)} If $\beta>0$, $\XX$ is compact, and \eref{phinlowbd} holds, then for any $M\ge 1$, $a_1,\cdots,a_{M+1}\in \RR$, $\C=\{y_1,\cdots,y_{M+1}\}\subset \XX$,
\be\label{coeffineq}
\sum_{k=1}^{M+1} |a_k| \le c\eta(\C)^{-\beta}\left\|\sum_{k=1}^{M+1} a_kG(\circ,y_k)\right\|_1.
\ee
\end{prop}

\begin{proof}\ 
The second inequality in \eref{btimesgnorm} is proved in \cite[Eqn.~(5.3)]{eignet}. The first inequality in \eref{btimesgnorm} follows easily from \cite[Eqn.~(5.11)]{eignet}.
Part (b) is proved in \cite[Proposition~5.2]{eignet}.
Part (c) is proved in \cite[Eqn.~(5.33)]{eignet}, used with $\gamma=0$.
Part (d) is proved in \cite[Theorem~3.4]{eignet}.
\end{proof}\\[1ex]

\noindent\textsc{Proof of Proposition~\ref{mumetricprop}.} \\

We note that Proposition~\ref{eignetprop} shows that $G$ is defined for \textbf{all} $x,y\in\XX$. Hence, \eref{glpuniform} shows that for any $\mu\in C_0^*$, $\int_\XX \|G(\circ,y)\|_pd|\mu|(y)$ is well defined, and hence, so is $\tn\mu\tn_{G;p}$.
%

It is clear that $\tn\circ\tn_{G;p}$ is a semi-norm. If $\mu\in C_0^*$ and $\tn\mu\tn_{G;p}=0$ then $b(\lambda_j)\hat{\mu}(j)=0$ for all $j$; i.e., $\hat{\mu}(j)=0$ for all $j$. Since the system $\{\phi_j\}_{j=0}^\infty$ is fundamental in $C_0$, this implies that $\mu=0$. The fact that $\tn\nu_n-\nu\tn_{G;p}\to 0$ implies that $\widehat{(\nu_n-\nu)}(k)\to 0$ for all $k$, which in turn implies that $\nu_n\stackrel{*}{\to}\nu$. Conversely, if $\nu_n\stackrel{*}{\to}\nu$ then for each $m\ge 0$,
$$
\lim_{n\to\infty}\int_\XX \Phi_{2^m}(hb_{2^m};x,y)d(\nu_n-\nu)(y)=0.
$$
The dominated convergence theorem now leads to the fact that $\tn\nu_n-\nu\tn_{G;p}\to 0$.  \qed\\

\noindent\textsc{Proof of Theorem~\ref{murecuptheo}.}\\

Using Fubini's theorem and then making a change of dummy variables, we see that for $x\in\XX$,
\bea\label{pf1eqn5}
\int_\XX G(x,y)d\nu_n(y)&=&\int_\XX G(x,y)\int_\XX\Phi_{2^n}(h;y,z)d\mu(z)d\mu^*(y)+\int_\XX G(x,y)\int_\XX\Phi_{2^n}(h;y,z)P_n(z)d\mu^*(z)d\mu^*(y) \nonumber\\
&=&
\int_\XX \Phi_{2^n}(hb_{2^n};x,y)d\mu(y)+\int_\XX \Phi_{2^n}(hb_{2^n};x,y)P_n(y)d\mu^*(y).
\eea
Hence, \eref{glpuniform} leads to
\bea\label{pf1eqn4}
\lefteqn{\left\| \int_\XX G(\circ,y)d\nu_n(y)-\int_\XX G(\circ,y)d\mu(y)\right\|_p}\nonumber\\
&\le& \int_\XX \|G(\circ,y)-\Phi_{2^n}(hb_{2^n};\circ,y)\|_pd|\mu|(y) + \int_\XX \|G(\circ,y)-\Phi_{2^n}(hb_{2^n};\circ,y)\|_p|P_n(y)|d\mu^*(y)\nonumber\\
&&\qquad\qquad+\int_\XX \|G(\circ,y)\|_p |P_n(y)|d\mu^*(y)\nonumber\\
& \le& c2^{-n(\beta-q/p')}\left\{|\mu|(\XX)+\|P_n\|_1\right\}+ c\|P_n\|_1.
\eea
This proves \eref{mudegapprox}. The proof of \eref{highpassapprox} is similar; the last term in the middle expression in \eref{pf1eqn4} does not appear in this case. \qed\\

It is convenient to organize some details of the proof of Theorem~\ref{convtheo} in the following lemma.

\begin{lemma}\label{convtheolemma}
Let $1\le p\le \infty$, $\mu\in C_0^*$, $\beta\in\RR$, and $b$ be a mask of type $\beta$. Then for any $r\in\RR$,
\be\label{betequiv}
\sup_{n\ge 1}2^{n(\beta+r)}\left\|\sigma_{2^n}(gb_{2^n};\mu)\right\|_p \sim \sup_{n\ge 1}2^{nr}\left\|\sigma_{2^n}(g;\mu)\right\|_p
\ee
\end{lemma}
\begin{proof}\ 
In view of \eref{btimesgnorm} with $p=1$, we have for any real $\beta$ and mask $b$ of type $\beta$.
$$
\sup_{x\in\XX}\int_\XX |\Phi_{2^n}(gb_{2^n}; x, y)|d\mu^*(y) \le c 2^{-n\beta}.
$$
Consequently, Young inequality shows that for $1\le p\le \infty$ and $f\in L^p$,
\be\label{pf2eqn1}
\left\|\sigma_n(gb_{2^n};f)\right\|_p \le c2^{-n\beta}\|f\|_p.
\ee
In this proof, let $\tilde{g}(t)=h(t/2)-h(4t)$.
Then $\tilde{g}$ is supported on $[1/8,2]\cup [-2,-1/8]$.
Analogous to \eref{pf2eqn1}, we see that for $m\ge 1$, \eref{btimesgnorm} and Young's inequality lead to
\be\label{opnormbd}
\|\sigma_{2^m}(g;f)\|_p \le c\|f\|_p.
\ee
In view of the fact that $g(t)\tilde{g}(t)=1$ for all $t$ in the support of $g$, we have
$$
\sigma_{2^n}(gb_{2^n};\mu)=\sum_{k=0}^\infty g\left(\frac{\lambda_k}{2^n}\right)b(\lambda_k)\hat{\mu}(k)\phi_k =
\sum_{k=0}^\infty g\left(\frac{\lambda_k}{2^n}\right)b(\lambda_k)\tilde{g}\left(\frac{\lambda_k}{2^n}\right)\hat{\mu}(k)\phi_k=\sigma_{2^n}(gb_{2^n};\sigma_{2^n}(\tilde{g};\mu)).
$$
Therefore, using \eref{pf2eqn1} with $\sigma_{2^n}(\tilde{g};\mu)$ in place of $f$, the fact that $\tilde{g}(t)=g(t/2)+g(t)+g(2t)$ for all $t\in\RR$, and \eref{opnormbd}, we conclude that
$$
\left\|\sigma_{2^n}(gb_{2^n};\mu)\right\|_p \le c2^{-n\beta}\left\|\sigma_{2^n}(\tilde{g};\mu)\right\|_p\le c2^{-n\beta}\left\{\left\|\sigma_{2^n}(g;\mu)\right\|_p+\left\|\sigma_{2^{n+1}}(g;\mu)\right\|_p+\left\|\sigma_{2^{n-1}}(g;\mu)\right\|_p\right\}.
$$
Hence,
$$
\sup_{n\ge 1}2^{n(\beta+r)}\left\|\sigma_{2^n}(gb_{2^n};\mu)\right\|_p \le c \sup_{n\ge 1}2^{nr}\left\|\sigma_{2^n}(g;\mu)\right\|_p.
$$
Since $1/b$ is a mask of type $-\beta$, this leads to \eref{betequiv}.
\end{proof}

\vskip 6pt
\noindent\textsc{Proof of Theorem~\ref{convtheo}.}\\
Let $f\in W_{r,p}$.
Since $g$ is supported on $[1/4,1]$, $\sigma_{2^m}(g;P)=0$ for all $P\in \Pi_{2^{m-2}}$. Therefore, using \eref{opnormbd}, we obtain for any $P\in \Pi_{2^{m-2}}$ that
$$
\|\sigma_{2^m}(g;f)\|_p=\|\sigma_{2^m}(g;f-P)\|_p \le c\|f-P\|_p;
$$
i.e.,
$$
\|\sigma_{2^m}(g;f)\|_p\le cE_{2^{m-2},p}(f).
$$
Since $f\in W_{r,p}$, this yields, together with \eref{betequiv} that
\be\label{pf2eqn3}
\left\|\sigma_{2^m}(gb_{2^m};f)\right\|_p \le c(f)2^{-m(\beta+r)}.
\ee
Therefore,
$$
\left\|\int_\XX G(\circ,y)f(y)d\mu^*(y)-\int_\XX G(\circ,y)\sigma_{2^n}(h;f)(y)d\mu^*(y)\right\|_p \le \sum_{m=n+1}^\infty \left\|\sigma_{2^m}(gb_{2^m};f)\right\|_p\le c(f)2^{-n(\beta+r)}.
$$
Thus, part (a) implies part (b).

Conversely, let part (b) hold. Then
\be\label{pf2eqn4}
\left\|\sigma_{2^m}(gb_{2^m};\mu)\right\|_p \le c(\mu)2^{-m(\beta+r)}, \qquad m\ge 1.
\ee
In view of \eref{betequiv}  this leads to
\be\label{pf2eqn5}
\|\sigma_{2^m}(g;\mu)\|_p \le c(\mu)2^{-mr}.
\ee
This implies that the sequence
$$
\sigma_{2^n}(h;\mu)=\sigma_1(h;\mu)+\sum_{m=0}^n \sigma_{2^m}(g;\mu)
$$
converges in $L^p$  to some $f\in L^p$. Moreover, $\hat{\mu}(k)=\hat{f}(k)$ for all $k=0,1,\cdots$. Therefore, $d\mu=fd\mu^*$.
Further, \eref{pf2eqn5} shows that
$$
E_{2^n,p}(f) \le \|f-\sigma_{2^n}(h;f)\|_p \le \sum_{m=n+1}^\infty \left\|\sigma_{2^m}(g;\mu)\right\|_p \le c(f)2^{-nr}.
$$
Thus, $f\in W_{r,p}$. \qed\\

Our proof of Theorems~\ref{widththeo} and \ref{discmasstheo} depends upon another notion of widths, the
so-called \emph{Bernstein width}. This is defined by for a weak-star compact subset $\mathcal{K}\subset C_0^*$ and integer $M\ge 1$ by
\be\label{bernwidth}
b_M(G,p;\mathcal{K})= \sup_{Y_{M+1}}\{\rho: \mu\in Y_{M+1}, \tn\mu\tn_{G;p}\le \rho \mbox{ implies } \mu\in \mathcal{K}\},
\ee
where the supremum is over all subspaces $Y_{M+1}$ of $C_0^*$ with dimension $M+1$.
It is proved in \cite[Theorem~3.1]{devore1989optimal} that for any integer $M\ge 1$,
\be\label{howardmicchtheo}
d_M(G,p;\mathcal{K})\ge b_M(G,p;\mathcal{K}).
\ee

 \noindent\textsc{Proof of Theorem~\ref{widththeo}.} \\

Let $\C=\{y_1,\cdots,y_n\}$ be a maximal $(\kappa_1(M+1))^{-1/q}$ separated subset of $\XX$, where $\kappa_1$ is the constant appearing in the upper bound \eref{ballmeasurecond} on the $\mu^*$-measure of balls. Then  $\disp\XX=\bigcup_{j=1}^n \mathbb{B}(y_j,(\kappa_1(M+1))^{-1/q})$. In view of  \eref{ballmeasurecond},
$$
1=\mu^*(\XX)\le \sum_{j=1}^n \mu^*\left(\mathbb{B}(y_j,(\kappa_1(M+1))^{-1/q})\right)\le n/(M+1).
$$
Therefore,  $\C_1=\{y_1,\cdots,y_{M+1}\}$ satisfies $\eta(\C_1)\ge  \eta(\C)\ge(\kappa_1(M+1))^{-1/q}$. We consider the $M+1$ dimensional  space
$Y=\mathsf{span}\{\delta_{y_k} : k=1,\cdots, M+1\}\subset C_0^*$, where $\delta_y$ denotes the Dirac delta at $y$. For any $\mu=\sum_{k=1}^{M+1}a_k\delta_{y_k}$,  Proposition~\ref{eignetprop}(d) shows that
$$
|\mu|(\XX)\le c\eta(\C_1)^{-\beta}\left\|\sum_{k=1}^{M+1} a_kG(\circ,y_k)\right\|_1\le cM^{\beta/q}\tn \mu\tn_{G;1}.
$$
In view of \eref{howardmicchtheo}, this leads to
\eref{widthest}. \qed\\

\noindent\textsc{Proof of Theorem~\ref{discmasstheo}.}\\

Choosing $M$ so that $(\kappa_1(M+1))^{-1/q}\ge\eta$, the elements of the space $Y$ constructed in the proof of Theorem~\ref{widththeo}
serves also for this theorem in order to apply \eref{howardmicchtheo}. \qed

\end{document}